\documentclass[9pt,a4paper]{amsart}
\usepackage{amsmath,amssymb,amsthm,mathtools,dsfont,bm,bbm,a4} 

\usepackage{color}
\usepackage{wasysym}
\newtheorem{lem}{Lemma}[section]
\newtheorem{prop}{Proposition}[section]

\newtheorem{thm}{Theorem}

\newtheorem*{assum}{Assumption (A)}

\theoremstyle{definition}

\theoremstyle{remark}

\theoremstyle{remark}

\newtheorem*{remarks*}{Remarks}
\newtheorem*{remark*}{Remark}
\numberwithin{equation}{section}

\newcommand{\C}{\mathbbm{C}}
\newcommand{\N}{\mathbbm{N}}

\newcommand{\R}{\mathbbm{R}}

\newcommand{\pt}{\partial}

\newcommand{\hD}{(-\Delta)^{1/2}}

\newcommand{\be}{\begin{equation}}
\newcommand{\ee}{\end{equation}}

\newcommand{\eps}{\epsilon}

\newcommand{\KK}{\mathcal{K}}

\newcommand{\ii}{\mathrm{i}}

\renewcommand{\rho}{\varrho}

\newcommand{\Le}{L^2_{\mathrm{even}}}

\newcommand{\Hil}{\mathsf{H}}

\catcode`@=11
\def\section{\@startsection{section}{1}%
  \z@{1.5\linespacing\@plus\linespacing}{.5\linespacing}%
  {\normalfont\bfseries\large\centering}}
\catcode`@=12

\makeatletter
\def\@cite#1#2{[\textbf{#1\if@tempswa , #2\fi}]}
\def\@biblabel#1{[\textbf{#1}]}
\makeatother

\begin{document}
\title[Uniqueness for Nonlocal Liouville Equation]{Uniqueness for the Nonlocal \\ Liouville Equation in $\R$}

\author{Maria Ahrend}
\address{M. Ahrend, Universit\"at Basel, Departement Mathematik und Informatik, Spiegelgasse 1, CH-4051 Basel, Switzerland}%
\email{maria.ahrend@unibas.ch}

\author{Enno Lenzmann}
\address{E. Lenzmann, Universit\"at Basel, Departement Mathematik und Informatik, Spiegelgasse 1, CH-4051 Basel, Switzerland}%
\email{enno.lenzmann@unibas.ch}

\maketitle


\begin{abstract}
We prove uniqueness of solutions for the nonlocal Liouville equation
$$
(-\Delta)^{1/2} w = K e^w \quad \mbox{in $\R$}
$$
with finite total $Q$-curvature $\int_{\R} K e^w \, dx< +\infty$. Here  the prescribed $Q$-curvature function $K=K(|x|) > 0$ is assumed to be a positive, symmetric-decreasing function satisfying suitable regularity and decay bounds. In particular, we obtain uniqueness of solutions in the Gaussian case with $K(x) = \exp(-x^2)$.

Our uniqueness proof exploits a connection of the nonlocal Liouville equation to ground state solitons for Calogero--Moser derivative NLS, which is a completely integrable PDE recently studied by P.~G\'erard and the second author.  \end{abstract}

\section{Introduction and Main Results}

\label{sec:intro}

In this paper, we consider the following nonlocal Liouville equation 
\be \label{eq:Liouville}
\hD w = K e^w \quad \mbox{in $\R$},
\ee
subject to the finiteness condition that 
\be \label{eq:Lambda}
\Lambda := \int_{\R} K(x) e^{w(x)} \, dx < +\infty.
\ee
Here $K : \R \to \R$ denotes a given function, which geometrically can be seen as a  \textbf{prescribed Q-curvature problem} in one space dimension. More precisely, if $w$ is a solution of \eqref{eq:Liouville} then $g=e^{2w} |dx|^2$ is a metric on $\R$ that is conformal to the standard metric $g_0=|dx|^2$ on $\R$ having constant $Q$-curvature equal to $1$. The quantity $\Lambda$ then corresponds to the  total $Q$-curvature of the metric $g$ on $\R$. We note that, by means of the stereographic projection, the nonlocal Liouville equation \eqref{eq:Liouville} can also be related to a prescribed $Q$-curvature problem on the unit circle. We refer to \cite{DaMa-17,DaMaRi-15} for more details on the geometric background on \eqref{eq:Liouville} and its relation to the generalized Riemann mapping theorem in the complex plane $\C$.

In fact, existence and non-existence results of prescribed $Q$-curvatures problems in $\R^n$ for general dimensions $n \geq 1$ have recently attracted a great deal of attention, leading to the class of Liouville type equations given by
\be  \label{eq:Liouville_general}
(-\Delta)^{n/2} w = K e^{n w} \quad \mbox{in $\R^n$}. 
\ee
In the case of $n=2$ space dimensions, equation \eqref{eq:Liouville} then becomes the well-known Liouville equation which is a central object in nonlinear elliptic PDEs and geometric analysis; see \cite{BrMe-91, ChYa-97, ChKi-95, Li-98, LiSh-94, Xu-05}. 

From an analytic point of view, a particularly challenging situation for equation \eqref{eq:Liouville_general} arises in odd space dimensions $n \in \{ 1,3,5,\ldots \}$ due to the nonlocal nature of the pseudo-differential operator $(-\Delta)^{n/2}$. Apart from the important special case of constant  positive constant $K>0$, where solutions $w$ are known in closed form (see \eqref{eq:w_explicit} below), the question of uniqueness of solutions $w$ have been entirely out of scope so far in odd dimensions. In the present paper, we address the case of $n=1$ space dimension. In fact, our analysis is strongly inspired by a surprisingly close connection to solitons of continuum limits of completely integrable {\em Calogero-Moser systems}; see below for more details on this.

Before we state our main results on \eqref{eq:Liouville}, we introduce some basic notions as follows. Throughout this paper, we assume that the solutions $w : \R \to \R$ belong to the space
$$
L_{1/2}(\R) := \Big \{ f \in L^1_{\mathrm{loc}}(\R) : \int_\R \frac{|f(x)|}{1+x^2} \, dx < +\infty \Big \},
$$
which is the natural space to define distributional solutions of \eqref{eq:Liouville}; see, e.\,g., \cite{HyMaMa-21} for details. We remark that we always deal with real-valued functions throughout this paper.

In order to state the main results, we will impose the following conditions on the $Q$-curvature function $K$, where we employ the commonly used short-hand notation $\langle x \rangle = \sqrt{1+x^2}$.

\begin{assum}
We assume that $K : \R \to \R$ has the following properties.
\begin{enumerate}
\item[(i)] $K$ is strictly positive, even and monotone decreasing in $|x|$.
\item[(ii)] $K$ is continuously differentiable and satisfies the pointwise bound
$$
\sqrt{K}(x) + |x \pt_x \sqrt{K}(x)| \leq C \langle x \rangle^{-1/2 - \delta} 
$$
for all $x \in \R$ with some constants $C >0$ and $\delta > 0$.
\end{enumerate}
\end{assum}

Important examples for admissible $Q$-curvatures are the Gaussian function $K(x) = e^{-x^2}$ and $K(x) = \langle x \rangle^{-1-2\delta}$ for some $\delta>0$. Further below, we will see that imposing regularity and decay conditions on the square root $\sqrt{K}$ of the $Q$-curvature function becomes natural due to our approach that is based on a connection to solitons for the Calogero-Moser derivative NLS discussed below.

The first main result of this paper is now as follows.
\begin{thm} \label{thm:main}
Suppose $K$ satisfies Assumption {\em \textbf{(A)}} and let $w \in L_{1/2}(\R)$ be a solution of \eqref{eq:Liouville} satisfying \eqref{eq:Lambda}. Then the following properties hold.
\begin{enumerate}
\item[(i)] \textbf{Regularity and Total $Q$-Curvature Bound:} We have $w \in C^{1,1/2}_{\mathrm{loc}}(\R)$ and the total $Q$-curvature $\Lambda = \int_\R K e^w \,dx$ satisfies
$0 < \Lambda < 2 \pi$.
\item[(ii)] \textbf{Symmetry and Monotonicity:}  $w$ is even and decreasing in $|x|$, i.\,e., it holds $w(-x) = w(x)$ for all $x \in \R$ and $w(x) \geq w(y)$ whenever $|x| \leq |y|$.
\item[(iii)] \textbf{Existence:} For every $w_0 \in \R$, there exists a solution $w \in L_{1/2}(\R)$ of \eqref{eq:Liouville} with $w(0)=w_0$ such that $\eqref{eq:Lambda}$ holds.
\end{enumerate}
\end{thm}

The next main result establishes uniqueness of solutions for the fractional Liouville equation \eqref{eq:Liouville}. In fact, despite the nonlocal nature of the problem, we obtain the following Cauchy--Lipschitz ODE type uniqueness result stating that the `initial value' $w(0)=w_0$ completely determines the solution $w$ in all of $\R$.  

\begin{thm}[\textbf{Uniqueness}] \label{thm:main2}
Suppose $K$ satisfies Assumption {\em \textbf{(A)}}. If $w, \widetilde{w} \in L_{\frac 1 2}(\R)$ are solutions of \eqref{eq:Liouville} satisfying \eqref{eq:Lambda}, then it holds 
$$
\widetilde{w}(0) = w(0) \quad \Rightarrow \quad \widetilde{w} \equiv w.
$$ 
\end{thm}

\begin{remarks*}
1) In view of existing techniques, we consider Theorem \ref{thm:main2} to be the most original contribution of the present paper. Further below, we will comment in more detail on the strategy behind its proof.

2) It remains an interesting open question whether -- instead of prescribing the `initial value' $w(0)$ -- we also have uniqueness of solutions $w$ determined by the value of the total $Q$-curvature $\Lambda$. That is, if for two solutions $\widetilde{w}, w \in L_{1/2}(\R)$ of \eqref{eq:Liouville} such that
$$
\int_\R K e^{\widetilde{w}} \, dx = \int_\R K e^w \, dx
$$
we necessarily have that the identity $\widetilde{w} \equiv w$ holds. We hope to address this open problem in the future.

3) We remark that the uniqueness result in Theorem \ref{thm:main2} is {\em non-perturbative}, since no smallness condition on either $K$ nor the `initial value' $w(0)$ is imposed. 

\end{remarks*}

\subsection*{Comments on the Uniqueness Proof}
Let us briefly describe the strategy behind proving the uniqueness result stated in Theorem \ref{thm:main2} above. The starting point rests on recasting the problem by introducing the positive function $v : \R \to \R_{>0}$ given by
\be \label{eq:v_w}
v = \sqrt{K e^w}  .
\ee
In terms of $v$, the nonlocal Liouville equation \eqref{eq:Liouville} acquires the form
\be \label{eq:vCM}
\pt_x v + W v + \frac{1}{2} \Hil(v^2) v = 0 \quad \mbox{in $\R$}.
\ee
Here $\Hil$ denotes the Hilbert transform on the real line and the function 
$$
W = -\pt_x (\log \sqrt{K})
$$
plays the role of a given external potential. In fact, equation \eqref{eq:vCM} and its solutions $v$ naturally arise in the study of \textbf{solitons} for the Calogero-Moser derivative NLS; see below. 

Despite the nonlocality of the Hilbert transform $\Hil$, it turns out that \eqref{eq:vCM} becomes more amenable to the study of uniqueness for solutions $v$ parametrized by its `initial value' $v_0 = v(0)$. To this end, we recast \eqref{eq:vCM} once more into a corresponding integral equation stated in \eqref{eq:v_int} below, where $v_0>0$ enters as a parameter. As a next essential step towards proving Theorem \ref{thm:main2}, we establish a local uniqueness result around any given solution $v$ of \eqref{eq:vCM} by constructing a locally unique branch parametrized by $v_0$ using the implicit function theorem. To achieve this, we show that the invertibility of the relevant linearized (and nonlocal) operator  is tantamount to ruling out non-trivial solutions $\psi \in \dot{H}^1_{\mathrm{even}}(\R)$ with $\psi(0)=0$ that satisfy
\be
(-\Delta)^{1/2} \psi - v^2 \psi = 0 \quad \mbox{in $\R$}.
\ee
Here the use of a \textbf{monotonicity formula} for the fractional Laplacian $(-\Delta)^{1/2}$ found in \cite{CaSo-05,CaSi-14} (and applied for the spectral analysis related to nonlinear ground states in \cite{FrLeSi-16}) becomes the key ingredient. However, in contrast to these works, we develop a different approach which completely avoids the use of the harmonic extension to the upper half-plane $\R^2_+$. Instead, we directly work with the singular integral expression for $(-\Delta)^{1/2}$ and we thus obtain expressions which relate to the classical theory of Carleman-Hankel operators on the half-line; see Section \ref{sec:uniqueness} and Appendix \ref{sec:mono} for more details. We believe  that this novel  approach for monotonicity formulas can lead to further general insights into spectral and uniqueness problems involving the fractional Laplacian $(-\Delta)^{s}$ with $s \in (0,1)$ and other suitable pseudo-differential operators (but which may not be seen as a Dirichlet-to-Neumann maps).

Once the local uniqueness for solutions $v$ of \eqref{eq:vCM} is established, we complete the proof of Theorem \ref{thm:main2} by a-priori bounds allowing us to make a global continuation argument linking to the limit $v_0 \to 0^+$, which finally shows that there exists only one global branch of solutions $v$ parametrized by $v(0)=v_0$.

\subsection*{Solitons for the Calogero-Moser Derivative NLS}
We now sketch the connection between the nonlocal Liouville equation \eqref{eq:Liouville} and solitons for the {\em Calogero-Moser derivative NLS}, which is a Hamiltonian PDE that can be written as 
\be \tag{CM} \label{eq:CM-NLS}
\ii \pt_t \psi = -\pt_{xx} \psi + V \psi - ((-\Delta)^{1/2} |\psi|^2) \psi + \frac{1}{4} |\psi|^4 \psi 
\ee
for the complex-valued field $\psi : [0,T) \times \R \to \C$.  For the external potential $V$, the choices 
$$ 
\mbox{$V(x) = x^2$ (external harmonic potential) and $V(x) \equiv 0$ (no external potential)}
$$ 
both arise naturally in the physical context of continuum limits of completely integrable many-body systems of Calogero-Moser type. For a formal derivation of \eqref{eq:CM-NLS} in the physics literature, we refer to \cite{AbBeWi-09,AbGrKu-11}. A striking feature, indicating that a completely integrable nature of the problem, is that \eqref{eq:CM-NLS} stems from a Hamiltonian energy functional $\mathcal{E}(\psi)$ which admits a factorization into a complete square of first-order terms; see \cite{GeLe-22, Ah-22}. More precisely, the Hamiltonian energy (up to an inessential additive term) is found to be
\be
\mathcal{E}(\psi) = \frac{1}{2} \int_{\R} \big |\pt_x \psi + \sqrt{V} \psi + \frac{1}{2} \Hil(|\psi|^2) \psi \big |^2 \, dx .
\ee
Minimizers of $\mathcal{E}(\psi)$ provide soliton solutions for \eqref{eq:vCM}; see \cite{GeLe-22, Ah-22} again for more details. Evidently,  solutions $v$ of \eqref{eq:vCM} above with $W = \sqrt{V}$ are minimizers for $\mathcal{E}(\psi)$. For an external harmonic potential when $V=x^2$, this corresponds to the choice of a prescribed $Q$-curvature in \eqref{eq:Liouville} given by the Gaussian function $K(x) = e^{-x^2}$. which clearly falls under the scope of Assumption \textbf{(A)}. 

In the case of no external potential $V \equiv 0$ and hence $K \equiv 1$ is a positive constant, it was recently shown in \cite{GeLe-22} by Hardy-space techniques that the real-valued minimizers $v \in H^1(\R)$ must be of the explicit form
\be
v(x) =  \pm \sqrt{\frac{2 \lambda}{1+\lambda^2(x-x_0)^2}}
\ee 
with arbitrary $\lambda >0$ and $x_0 \in \R$. Translating this back via \eqref{eq:v_w} and using regularity theory, this shows that all solutions $w \in L_{1/2}(\R)$ of the nonlocal Liouville equation \eqref{eq:Liouville} with $K \equiv 1$ and $e^w \in L^1(\R)$ are given by
\be \label{eq:w_explicit}
w(x) = \log \left ( \frac{2 \lambda}{1+\lambda^2 (x-x_0)^2} \right ).
\ee
This uniqueness result for \eqref{eq:Liouville} with constant $K \equiv 1$ has also been obtained in \cite{ChYa-97,Xu-05,DaMa-17} by different techniques. However, the approach taken in \cite{GeLe-22} provides yet another self-contained and independent proof of this fact by exploiting the relation to solitons for \eqref{eq:CM-NLS}.

\subsection*{Acknowledgments} E.~L.~gratefully acknowledges financial support by the Swiss National Science Foundation (SNSF) under Grant No.~204121.

\section{Preliminaries}

We first collect some results that can be deduced by adapting known arguments. In particular, the results in this section will imply that item (i) in Theorem \ref{thm:main} holds true.

Throughout this section, we always assume that $w \in L_{1/2}(\R)$ solves \eqref{eq:Liouville} subject to \eqref{eq:Lambda}, where the $Q$-curvature function $K$ satisfies Assumption \textbf{(A)}. 

\subsection{Regularity, Asymptotics, and Universal Bound on $\Lambda$}

We start by collecting some immediate facts about the $Q$-curvature function $K$ satisfying Assumption \textbf{(A)}.

\begin{lem} \label{lem:K_prop}
It holds that $\sqrt{K}, K \in H^1(\R)$ and $K \in L^1(\R) \cap L^\infty(\R)$.
\end{lem}

\begin{proof}
Since $\sqrt{K} \lesssim \langle x \rangle^{-1/2-\delta}$ for some $\delta > 0$, we readily see that $\sqrt{K} \in L^2(\R) \cap L^\infty(\R)$, whence it follows that $K \in L^1(\R) \cap L^\infty(\R)$. Furthermore, by the bound $|x \pt_x \sqrt{K}(x)| \lesssim \langle x \rangle^{-1/2-\delta}$ for some $\delta > 0$ and the fact that $\pt_x \sqrt{K}$ is continuous and hence locally bounded, we deduce that $\pt_x \sqrt{K} \in L^2(\R)$ holds. Thus $\pt_x K = 2  \sqrt{K} \pt_x \sqrt{K}\in L^2(\R)$ since $\sqrt{K} \in L^\infty(\R)$. This shows that $\sqrt{K}$ and $K$ both belong to $H^1(\R)$.
\end{proof}

Next, we derive the following regularity result for solutions $w \in L_{\frac 1 2}(\R)$ of \eqref{eq:Liouville} subject to the integrability condition $K e^w \in L^1(\R)$.

\begin{lem} \label{lem:reg_w}
It holds that $w \in C^{1,1/2}_{\mathrm{loc}}(\R)$. 
\end{lem} 

\begin{proof}
We can adapt the arguments presented in \cite{HyMaMa-21}, where regularity for the equation $(-\Delta)^{n/2} u = |x|^{n \alpha} e^{nu}$ in $\R^n$ with $\alpha > -1$ subject to the integrability condition $\int_{\R^n} |x|^{n\alpha} e^{nu} < +\infty$ is discussed. 

For the reader's convenience, we sketch the necessary modifications for our case. First, we show that $e^w \in L^p_{\mathrm{loc}}(\R)$ for any $p \in [1, \infty)$ by an `$\eps$-regularity trick' as follows. Indeed, for any such $p \geq 1$, we can take $0 < \eps < \frac{\pi}{p}$ and we split $K e^w = f_1 + f_2$ with $f_1, f_2 \geq 0$ such that $f_1 \in L^1(\R) \cap L^\infty(\R)$ and $\| f_2 \|_{L^1} \leq \eps$. Next, we write
$$
w = w_1 + w_2 + w_3
$$
with the functions
$$
w_i(x) := \frac{1}{\pi} \int_{\R} \log \left ( \frac{1+|y|}{|x-y|} \right ) f_i(y) \, dy \quad \mbox{for $i=1,2$}, \quad w_3 := w-w_1-w_2.
$$
We have that $w_1 \in C^0(\R)$ and $w_3 \in C^\infty(\R)$ since $w_3$ is $\frac{1}{2}$-harmonic. (In fact, we can deduce that $w_3$ must be a constant function; see also the proof of Lemma \ref{lem:w_int_log} below.) For any $R >0$ be given, we apply Jensen's inequality to find
\begin{align*}
\int_{B_R} e^{p w_2} \, dx & = \int_{B_R} \exp \left ( \int_{\R} \frac{p \| f_2 \|_{L^1}}{\pi} \log \left ( \frac{1+|y|}{|x-y|} \right ) \frac{f_2(y)}{\|f_2 \|_{L^1}} \,dy \right )  dx \\
& \leq \int_{B_R} \int_{\R} \exp \left ( \frac{p \| f_2 \|_{L^1}}{\pi} \log \left ( \frac{1+|y|}{|x-y|} \right ) \right ) \frac{f_2(y)}{\|f_2 \|_{L^1}} \, dy \, dx \\
& = \frac{1}{\|f_2 \|_{L^1}} \int_\R f_2(y) \int_{B_R} \left ( \frac{1+|y|}{|x-y|} \right )^{\frac{p \|f_2 \|_{L^1}}{\pi}} \,dx \, dy < +\infty,
\end{align*}
since $p \| f_2 \|_{L^1} \leq p \eps < \pi$ holds. This shows that $e^{w_2} \in L^p_{\mathrm{loc}}(\R)$ and hence $e^w \in L^p_{\mathrm{loc}}(\R)$ by the regularity of $w_1$ and $w_3$.

Since $K \in L^\infty(\R)$ and thus $K e^w \in L^p_{\mathrm{loc}}(\R)$, we see that $w \in W^{1,p}_{\mathrm{loc}}(\R)$ for all $p \in (1, \infty)$.  By Sobolev embeddings, this implies the H\"older continuity $w \in C^{0,\alpha}_{\mathrm{loc}}(\R)$ for any $\alpha \in (0,1)$. Recall that $K \in H^1(\R) \subset C^{0,1/2}(\R)$, whence it follows that $K e^w \in C^{0,1/2}_{\mathrm{loc}}(\R)$. By local Schauder-type estimates for fractional Laplacians, we deduce that $w \in C^{1,1/2}_{\mathrm{loc}}(\R)$.
\end{proof}

\begin{lem} \label{lem:w_int_log}
It holds that
$$
w(x) = \frac{1}{\pi} \int_{\R} \log \left ( \frac{1+|y|}{|x-y|} \right ) K(y) e^{w(y)} \, dy + C
$$
with some constant $C \in \R$. Moreover, we have $K e^w \in L^1(\R) \cap L^\infty(\R)$ and the asymptotics 
$$
w(x) = -\frac{\Lambda}{\pi} \log |x| + O(1) \quad \mbox{as} \quad  |x| \to +\infty
$$
holds, where $\Lambda = \int_{\R} K e^w \, dx >0$. Finally, we have that $w_+ = \max \{w,0 \} \in L^\infty(\R)$ and
$$
\int_{\R} \log (1 +|x|) K(x) e^{w(x)} \, dx < +\infty.
$$
\end{lem}

\begin{proof}
Define the function $\tilde{w}(x) := w(x)-\frac{1}{\pi} \int_{\R} \log \left ( \frac{1+|y|}{|x-y|} \right ) K(y) e^{w(y)} \, dy$. Then $\tilde{w} \in L_{1/2}(\R)$ satisfies $(-\Delta)^{1/2} \tilde{w} = 0$ in $\R$. From \cite{Fa-16} we deduce that  $\tilde{w} : \R \to \R$ is an affine function and thus $\tilde{w} = \mbox{const}.$, since we have $w \in L_{1/2}(\R)$. This proves the integral formula for $w$ stated above.

From the discussion in \cite{HyMaMa-21}[Remark 3.2] we deduce that
$$
\lim_{|x| \to +\infty} \frac{w(x)}{\log |x|} = -\frac{\Lambda}{\pi}
$$
with $\Lambda = \int_{\R} K e^w \, dx > 0$. Clearly, the above limit implies that $w(x) < 0$ for $|x| \geq R$ with $R >0$ sufficiently large. Since $w \in L^\infty_{\mathrm{loc}}(\R)$ by Lemma \ref{lem:reg_w}, we thus find $w^+ = \max\{w,0 \} \in L^\infty(\R)$ and hence $e^w \in L^\infty(\R)$. By our assumptions on $K$, this implies
$$
0 < K(x) e^{w(x)} \leq C \langle x \rangle^{-1-2\delta} \in L^1(\R) \cap L^\infty(\R)
$$
with some $\delta> 0$. Thus the function $f := Ke^w$ satisfies $\log(1+|\cdot|) f \in L^1(\R)$. Hence we can rewrite the integral formula for $w$  as
$$
w(x) = -\frac{1}{\pi} \int_{\R} \log |x-y| f(y) \,dy + C = -\frac{\Lambda}{\pi} \log |x| - \frac{1}{\pi} \int_\R \log \left | \frac{x-y}{x} \right | f(y) \, dy + C
$$
with some constant $C \in \R$. The asymptotic formula for $w$ now follows from
\be \label{eq:DCT}
\lim_{|x| \to +\infty} \int_\R  \log \left | \frac{x-y}{x} \right | f(y) \, dy = 0.
\ee
Indeed, this can be seen by splitting the integration into the sets $\{ |x-y| \geq |x|/2\}$ and $\{ |x-y| \leq |x|/2\}$ and by using that $f \in  L^1(\R; (1+\log(1+|x|)) dx) \cap L^2(\R)$ and dominated convergence. We omit the details.
\end{proof}

\begin{lem} \label{lem:poho}
The total $Q$-curvature satisfies $0 < \Lambda < 2\pi$.
\end{lem}

\begin{proof} 
By the positivity of $K(x)>0$, it is clear that $\Lambda >0$ holds. The upper bound $\Lambda < 2 \pi$ follows by Pohozaev-type argument adapted to Liouville type equations; see, e.\,g., \cite{Xu-05,DaMaRi-15}. For the reader's convenience, we state the proof adapted to our case. 

If we differentiate the integral equation for $w$ and multiply with $x K(x)$, we obtain
\be \label{eq:poho}
x K(x) \frac{\pt w}{\pt x} = -\frac{1}{\pi} PV \int_{\R} \frac{x K(x)}{x-y} K(y) e^{w(y)} \, dy .
\ee
Multiplication with $e^{w(x)}$ and integration of the left-hand side over $[-R,R]$ yields
\begin{align*}
I & := \int_{-R}^R x K(x) e^{w(x)} \frac{\pt w}{\pt x} \, dx = x K(x) e^{w(x)} \Big |_{x=-R}^R - \int_{-R}^{R} \frac{\pt}{\pt x} (xK(x)) e^{w(x)} \, dx \\
& \to -\Lambda  - \int_{\R} x (\pt_x K)(x) e^{w(x)} \, dx \quad \mbox{as} \quad R \to +\infty.
\end{align*}
Note that $x K(x) e^{w(x)} \Big |_{x=-R}^R \to 0$ as $R \to \infty$ since $|x K(x) e^{w(x)}| \leq C \langle x \rangle^{-2\delta}$ for some $\delta > 0$ in view of $e^{w} \in L^\infty(\R)$ and our assumption on $K$. Furthermore, we notice that $x (\pt_x K) e^{w(x)} \in L^1(\R)$ by the assumption on $K$.

On the other hand, if we use the right-hand side in \eqref{eq:poho} we deduce
\begin{align*}
II & := -\frac{1}{\pi} \int_{-R}^R \int_{\R} \frac{x K(x)}{x-y} e^{w(x)} K(y) e^{w(y)} \, dy \, dx = -\frac{1}{2 \pi} \int_{-R}^R \int_{\R} K(x) e^{w(x)} K(y) e^{w(y)} \, dy \, dx \\
& \quad - \frac{1}{2 \pi} \int_{-R}^R \int_\R \frac{x+y}{x-y} K(x) e^{w(x)} K(y) e^{w(y)} \, dy \, dx  \to -\frac{\Lambda^2}{2 \pi} + 0 \quad \mbox{as} \quad R \to \infty.
\end{align*}
Since $I=II$, we deduce that
$$
\frac{\Lambda}{2\pi} (\Lambda - 2 \pi)  = \int_{\R} (x \pt_x K) e^{w} \, dx .
$$
Because $K$ is monotone decreasing in $|x|$ and non-constant, we see that $\int_{\R} (x \pt_x K) e^w \, dx < 0$. This implies that $\Lambda < 2 \pi$ must hold.
\end{proof}

\subsection*{Proof of Theorem \ref{thm:main} (i)} The proof directly follows from Lemmas \ref{lem:reg_w} and \ref{lem:poho}. \hfill  $\Box$

\section{Even Symmetry}

This section is devoted to the proof of item (ii) in Theorem \ref{thm:main}. We implement the method of moving planes; actually, it is a `moving point' argument since we are in one space dimension. Because of the nonlocal nature of the problem, it is expedient to work with the equation for $w(x)$ written in integral form. We then adapt the moving plane method generalized to integral equations, which was initiated in the work of \cite{ChLiOu-06}.

For $\lambda > 0$ and $x \in \R$, we set
$$
\Sigma_\lambda := [\lambda, \infty), \quad x_\lambda := 2 \lambda - x, \quad w_\lambda(x) := w(x_\lambda), \quad \Sigma^w_\lambda := \{ x \in \Sigma_\lambda : w(x) > w_\lambda(x) \}.
$$
From the proof of Lemma \ref{lem:w_int_log} we recall the integral representation
\be \label{eq:w_int_log}
w(x) = \int_{\R} G(x-y) K(y) e^{w(y)} \, dy + C
\ee
with some constant $C \in \R$ and we denote
$$
G(x) = -\frac{1}{\pi} \log |x| .
$$

We first initiate the moving plane method by showing that $\Sigma_\lambda^w$ is empty in the regime of sufficiently large $\lambda$.

\begin{prop} \label{prop:mov1}
There exists $\lambda_0 > 0$ such that $\Sigma^w_\lambda = \emptyset$ for all $\lambda > \lambda_0$.
\end{prop}

\begin{proof}
Let $\lambda > 0$. Using \eqref{eq:w_int_log} and the fact $G$ is an even function, a calculation yields 
$$
w(x) - w_\lambda(x) = \int_{\Sigma_\lambda} \left ( G(x-y) - G(x_\lambda-y) \right ) \left ( K(y) e^{w(y)} - K(y_\lambda) e^{w_\lambda(y)} \right ) \, dy.
$$
Since $G(x)$ and $K(x)$ are monotone decreasing functions in $|x|$, we deduce 
$$
G(x-y) \geq G(x_\lambda-y) \quad \mbox{and} \quad 0 < K(y) \leq K(y_\lambda)  \quad \mbox{for $x,y \in \Sigma_\lambda$}.
$$
For any $x \in \Sigma^w_\lambda \subset \Sigma_\lambda$, we thus estimate
\begin{align*}
w(x) - w_\lambda(x) & \leq \int_{\Sigma_\lambda} \left ( G(x-y) - G(x_\lambda-y) \right ) K(y) \left ( e^{w(y)} - e^{w_\lambda(y)} \right ) dy \\
& \leq \int_{\Sigma^w_\lambda}   \left ( G(x-y) - G(x_\lambda-y) \right ) K(y) \left ( e^{w(y)} - e^{w_\lambda(y)} \right ) dy \\
& \leq \int_{\Sigma^w_\lambda}   \left ( G(x-y) - G(x_\lambda-y) \right ) F_\lambda(y) \, dy, 
\end{align*}
where we denote
$$
F_\lambda(y) := K(y) e^{w(y)} \left ( w(y) - w_\lambda(y) \right )  .
$$
Note that $F_\lambda(y) > 0$ for $y \in \Sigma_\lambda^w$. Next, we observe the upper bounds 
$$
-G(x_\lambda-y) = \frac{1}{\pi} \log |2 \lambda - x -y| \leq \frac{1}{\pi} \left ( \log 2 + \left |\log x \right | + \left |\log y \right |\right )\quad \mbox{for $x,y \in \Sigma_{\lambda}$}, 
$$
and
$$
G(x-y) \leq 0 \quad \mbox{for $|x-y| \geq 1$}.
$$
Thus we find, for $x \in \Sigma^w_\lambda$,
\begin{align*}
0 < w(x) - w_\lambda(x) & \leq \int_{\Sigma^w_\lambda \cap \{|x-y| \leq 1\}} G(x-y) F_\lambda(y) \, dy  \\
& \quad + \frac{1}{\pi} \int_{\Sigma^w_\lambda} \left (\log 2 + \left |\log x \right | + \left |\log y \right | \right ) F_\lambda(y) \, dy.
\end{align*}
Next, we let $\alpha := 1 + \delta >1$ with $\delta > 0$ taken from Assumption \textbf{(A)}. Thus we deduce
\begin{align*}
\left \| \langle x \rangle^{-\alpha} (w-w_\lambda) \right \|_{L^1(\Sigma^w_\lambda)} & \leq \int_{\Sigma_\lambda^w} \left ( \int_{\Sigma^w_\lambda \cap \{ |x-y| \leq 1 \}} \langle x \rangle^{-\alpha} G(x-y) F_\lambda(y) \, dy \right )  dx \\
& \quad + \frac{1}{\pi} \int_{\Sigma_\lambda^w} \left ( \int_{\Sigma^w_\lambda} \langle x \rangle^{-\alpha}   \left ( \log 2 + \left |\log x \right | + \left |\log y \right |\right ) F_\lambda(y) \, dy \right ) dx  \\ 
& = \int_{\Sigma^w_\lambda} \left \{ C_{1, \Sigma^w_\lambda}(y) + C_{2, \Sigma^w_\lambda}(y) \right \}  F_\lambda(y) \, dy,
\end{align*}
where we have
\be \label{def:C1}
C_{1,\Sigma^w_\lambda} := \int_{\Sigma^w_\lambda \cap \{ |x-y| \leq 1 \}} \langle x \rangle^{-\alpha} G(x-y) \, dx \leq \int_{y-1}^{y+1} G(x-y) \, dx = \frac{2}{\pi},
\ee
\be \label{def:C2}
C_{2, \Sigma^w_\lambda} := \frac{1}{\pi} \int_{\Sigma^w_\lambda} \langle x \rangle^{-\alpha}  \left ( \log 2 + \left |\log x \right | + \left |\log y \right |\right )  dx \leq C ( \left | \log y \right | + 1) 
\ee 
with some constant $C=C(\alpha) > 0$ independent of $\lambda> 0$. Therefore, we have found that 
\begin{align*}
\left \| \langle x \rangle^{-\alpha} (w-w_\lambda) \right \|_{L^1(\Sigma^w_\lambda)} & \leq  C  \int_{\Sigma^w_\lambda} (\left | \log y \right | +1) K(y) e^{w(y)} (w(y) - w_\lambda(y)) \, dy \\
& \leq C   \sup_{y \geq \lambda} \left ( (\left | \log y \right | + 1) \langle y \rangle^{\alpha} K(y) e^{w(y)} \right )  \left \| \langle x \rangle^{-\alpha} (w-w_\lambda) \right \|_{L^1(\Sigma_\lambda^w)} 
\end{align*}
where the constant $C> 0$ is independent of $\lambda$. Since $e^w \in L^\infty(\R)$ and by Assumption \textbf{(A)} we have $K \leq C\langle x \rangle^{-1-\delta}$, there exists a constant $\lambda_0 > 0$ sufficiently large such that
$$
\left \| \langle x \rangle^{-\alpha} (w-w_\lambda) \right \|_{L^1(\Sigma^w_\lambda)} \leq \frac{1}{2} \left \| \langle x \rangle^{-\alpha} (w-w_\lambda) \right \|_{L^1(\Sigma^w_\lambda)} \quad \mbox{for $\lambda > \lambda_0$}.
$$
Thus the set $\Sigma_\lambda^w$ has measure zero for $\lambda > \lambda_0$, which implies that $\Sigma_\lambda^w = \emptyset$ for $\lambda > \lambda_0$  by continuity of $w-w_\lambda$.
\end{proof}

As a next step, we establish the following continuation property.

\begin{prop} \label{prop:mov2}
Suppose that $\lambda_0 > 0$ satisfies $\Sigma^w_\lambda = \emptyset$ for all $\lambda > \lambda_0$. Then there exists $\eps > 0$ such that $\Sigma_\lambda^w = \emptyset$ for all $\lambda > \lambda_0 -\eps$.
\end{prop}

\begin{proof}
We divide the proof into the following steps.

\medskip
\textbf{Step 1.} By assumption, we have $w(x) \leq w_\lambda(x)$ for all $x \geq \lambda$ and $\lambda > \lambda_0$. By continuity, we conclude that $w(x) \leq w_{\lambda_0}(x)$ for all $x \geq \lambda_0$. This shows that $\Sigma^w_{\lambda_0} = \emptyset$ holds.

To show that we indeed have $\Sigma^w_\lambda = \emptyset$ for any $\lambda > \lambda_0-\eps$ with some $\eps > 0$, we argue as follows. First, we claim that the strict inequality holds:
\be \label{ineq:w_strict}
w(x) < w_{\lambda_0}(x) \quad \mbox{for all $x > \lambda_0$}.
\ee
We argue by contradiction. Suppose that $w(x) = w_{\lambda_0}(x)$ for some $x  > \lambda_0$. Since $K(x)$ is monotone decreasing in $|x|$, we find
$$
(-\Delta)^{1/2} ( w_{\lambda_0}-w)(x)= \left (K(x_{\lambda_0})-K(x) \right ) e^{w(x)} \geq 0.
$$
On the other hand, by using the singular integral expression for $(-\Delta)^{1/2}$, we conclude
\begin{align*}
(-\Delta)^{1/2} ( w_{\lambda_0}-w)(x) &= \frac{1}{\pi} PV \int_{\R} \frac{(w_{\lambda_0}-w)(x) - (w_{\lambda_0}-w)(y)}{(x-y)^2} \, dy \\
& = -\frac{1}{\pi} PV \int_{\R} \frac{(w_{\lambda_0}-w)(y)}{(x-y)^2} \, dy \\
& = -\frac{1}{\pi} PV \int_{\lambda_0}^\infty \left ( \frac{1}{(x-y)^2} - \frac{1}{(x-y_{\lambda_0})^2} \right ) (w_{\lambda_0}-w)(y) \, dy \leq 0.
\end{align*}
Since we must have equality and $x > \lambda_0$, we deduce that $w_{\lambda_0} - w \equiv 0$ on $\Sigma_{\lambda_0}$. Therefore,
$$
0 = (-\Delta)^{1/2} (w_{\lambda_0} - w)(y) = \left ( K(y_{\lambda_0})- K(y) \right ) e^{w(y)} \quad \mbox{for $y \in \Sigma_{\lambda_0}$}.
$$
Thus $K(y) = K(y_{\lambda_0})$ for every $y \in \Sigma_{\lambda_0}$, which means that $K : \R \to \R$ is symmetric with respect to reflection at $\{ y = \lambda_0\}$ with some $\lambda_0 > 0$. Since $K$ is also symmetric with respect to the origin by Assumption \textbf{(A)}, we conclude that $K$ is constant. But this contradicts Assumption \textbf{(A)}. This completes our proof of claim \eqref{ineq:w_strict}.

\medskip
\textbf{Step 2.}
From the proof of Proposition \ref{prop:mov1} we recall the  estimate
$$
\left \| \langle x \rangle^{-\alpha} (w-w_\lambda) \right \|_{L^1(\Sigma^w_\lambda)} \leq C_{\Sigma^w_\lambda}  \sup_{y \geq \lambda} \left ( (\left | \log y \right | + 1) \langle y \rangle^{\alpha} K(y) e^{w(y)} \right )  \left \| \langle x \rangle^{-\alpha} (w-w_\lambda) \right \|_{L^1(\Sigma_\lambda^w)} .
$$
Now, in view of \eqref{ineq:w_strict}, we conclude that the set
$$
\overline{\Sigma^w_{\lambda_0}} := \{ x \in \Sigma_{\lambda_0} : w(x) \geq w_{\lambda_0}(x) \} = \{ x = \lambda_0\}
$$
has Lebesgue measure zero. Since $\lim_{\lambda \nearrow \lambda_0} \Sigma^w_{\lambda} \subset \overline{\Sigma^w_{\lambda_0}}$ and, by inspecting the expression for $C_{\Sigma^w_\lambda}$ (see \eqref{def:C1} and \eqref{def:C2}), we deduce that $C_{\Sigma^w_\lambda} \to 0$ as $\lambda \nearrow \lambda_0$. Thus for some $\eps > 0$ sufficiently small, we find that
$$
\left \| \langle x \rangle^{-\alpha} (w-w_\lambda) \right \|_{L^1(\Sigma^w_\lambda)} \leq \frac{1}{2} \left \| \langle x \rangle^{-\alpha} (w-w_\lambda) \right \|_{L^1(\Sigma^w_\lambda)} \quad \mbox{for $\lambda > \lambda_0 -\eps$},
$$
which implies that $\Sigma^w_\lambda$ as Lebesgue measure zero and hence is empty for $\lambda > \lambda_0-\eps$ by the continuity of $w-w_\lambda$.

This completes the proof of Proposition \ref{prop:mov2}.
\end{proof}

Finally, we show the following closedness property.

\begin{prop} \label{prop:mov3}
If $\Sigma^w_{\lambda} = \emptyset$ for all $\lambda > \lambda_n$ and $\lambda_n \to \lambda_*$, then $\Sigma^w_{\lambda} = \emptyset$ for all $\lambda > \lambda_*$. 
\end{prop}

\begin{proof}
This property holds trivially by construction. Indeed, let $\lambda > \lambda_*$ be arbitrary. Since $\lambda_n \to \lambda_*$, there exists $n \in \N$ such that $\lambda > \lambda_n$ and hence $\Sigma^w_{\lambda} = \emptyset$.
\end{proof}

\subsection*{Proof of Theorem \ref{thm:main} (ii): Symmetry}
We are now ready to give the proof of Theorem \ref{thm:main}, item (ii). By combining Propositions \ref{prop:mov1}--\ref{prop:mov3} and using a standard open-closed argument, we deduce that $\Sigma^w_\lambda = \emptyset$ for all $\lambda > 0$. By continuity of $w$ and passing to the limit $\lambda \to 0^+$, we thus obtain
$$
w(x) \leq w(-x) \quad \mbox{for all $x \geq 0$}.
$$
On the other hand, replacing $w(x)$ by $\tilde{w}(x)=w(-x)$ yields another solution of \eqref{eq:Liouville}. By re-running the moving plane argument above, we find the opposite inequality $w(-x) \leq w(x)$ for $x \geq 0$. This shows that $w(x) = w(-x)$ for all $x \in \R$.   

Finally, we show that $w$ is a decreasing function of $|x|$. Since $w$ is even, it suffices to prove that $w(y) \leq w(x)$ for $y >x \geq 0$. In fact, let $y > x \geq 0$ and define $\lambda = \frac{x+y}{2} > 0$, which directly yields $x= 2 \lambda -y=y_\lambda$ and $y \in \Sigma_\lambda$. Since $\Sigma^w_\lambda = \emptyset$, we see that $w(y) \leq w_\lambda(y) = w(y_\lambda)=w(x)$.  

The proof of Theorem \ref{thm:main} (ii) is now complete. \hfill $\Box$

\section{Compactness and A-Priori Estimates}

In this section, we derive results which will be used to prove Theorem \ref{thm:main2} (iii) about existence of solutions. Some estimates will also be later needed in the rest of the proof of Theorem \ref{thm:main2} below.
\subsection{Preliminaries}
\label{subsec:prelim}
We start by recasting the nonlocal Liouville equation \eqref{eq:Liouville} into a more convenient form as follows. Recall that we will always assume that the $Q$-curvature function $K(x)> 0$ satisfies Assumption \textbf{(A)}. 

Suppose that $w \in L_{\frac 1 2}(\R)$ is a solution of \eqref{eq:Liouville} with $K e^w \in L^1(\R)$. We define the positive function $\rho : \R \to \R_{>0}$ by setting
\be 
\rho(x) = K(x) e^{w(x)} .
\ee
From Lemma \ref{lem:reg_w} and \ref{lem:w_int_log} we see that $\rho$ is of class $C^1$ and $\rho \in L^1(\R) \cap L^\infty(\R)$. Furthermore, by applying the Hilbert transform $\Hil$ on both sides of \eqref{eq:Liouville} and using that $\Hil (-\Delta)^{\frac 1 2} = -\pt_x$, we obtain 
\be 
-\pt_x w = \Hil(\rho).
\ee
Since $w$ is $C^1$, this is equivalent to the integral equation
\be
w(x) = w(0) - \int_0^x \Hil(\rho)(y) \, dy .
\ee
Using that $\rho(x) = K(x) e^{w(x)}$ and $K(x)> 0$, we can rewrite this equation as
\be
\rho(x) = \rho(0) \frac{K(x)}{K(0)} e^{-\int_0^x  \Hil(\rho)(y) \, dy } .
\ee
Next, by using that $\rho(x) > 0$ is a positive $C^1$-function, we can introduce the positive $C^1$-function $v : \R \to \R_{>0}$ with $v(x) = \sqrt{\rho(x)}$. By taking the square root in the previous equation, we obtain
\be \label{eq:v_int}
\boxed{v(x) = v(0) \sqrt{\frac{K(x)}{K(0)}} e^{-\frac{1}{2} \int_0^x \Hil(v^2)(y) \, dy }}
\ee
By differentiating, we readily check that $v$ solves the nonlinear equation
\be \label{eq:v_diff}
\pt_x v - (\pt_x \log \sqrt{K}) v + \frac{1}{2} \Hil(v^2) v = 0.
\ee
We remark that, for the special case $K(x)=e^{-x^2}$, we retrieve the ground state soliton equation for the harmonic Calogero--Moser DNLS. Of course, the following analysis will allow for more general $K(x)$ that satisfy Assumption \textbf{(A)}. We also note that $v = \sqrt{\rho} \in L^2(\R)$. Since $w$ is even and decreasing in $|x|$, the same holds for $v(x) = \sqrt{\rho(x)} = \sqrt{K(x) e^{w(x)}}$ due to the symmetric-decreasing property of $K(x)$. Since $v=v^*$ is also symmetric-decreasing, we deduce $e^{-\frac{1}{2} \int_0^x \Hil(v^2) \,dy} \leq 1$ for all $x \in \R$ by Lemma \ref{lem:hilbert}. Now, a glimpse at \eqref{eq:v_diff} yields the bound
\begin{align*}
\| \pt_x v \|_{L^2} & \lesssim |v(0)| \| (\pt_x \sqrt{K}) e^{-\frac{1}{2} \int_0^x \Hil(v^2)} \|_{L^2} + \frac{1}{2} \| v \Hil(v^2) \|_{L^2} \\
& \lesssim |v(0)|\| \pt_x \sqrt{K} \|_{L^2} + \| v \|_{L^\infty} \| v^2 \|_{L^2}  \lesssim |v(0)| ( 1 + \| v \|_{L^4}^2) < +\infty,
\end{align*}
by also using $v = \sqrt{\rho}  \in L^4(\R)$ in view of the fact that $\rho \in L^1(\R) \cap L^\infty(\R)$. This show that $v \in H^1(\R)$ holds. Moreover, by Lemma \ref{lem:w_int_log}, we notice that 
$$
\int_\R \log (1 + |x|) |v(x)|^2 \, dx < +\infty.
$$

On the other hand, it is easy to see that any solution $v \in H^1(\R)$ with $v(0)> 0$ of \eqref{eq:v_int} is automatically $C^1$ and gives rise to a solution $w\in L_{1/2}(\R)$ of \eqref{eq:Liouville} by setting $w = \log(K^{-1} v^2)$ satisfying the finiteness condition $\int_{\R} K e^w \, dx = \int_{\R} v^2 \,dx < +\infty$. In summary and view of the symmetry result shown in Theorem \ref{thm:main} (ii), we have established the following result.

\begin{prop} \label{prop:wv_equiv}
A function $w \in L_{\frac 1 2}(\R)$ solves \eqref{eq:Liouville} with $K e^w \in L^1(\R)$ if and only if $v = \sqrt{K e^w} \in H^1(\R)$ solves \eqref{eq:v_int}. Moreover, we have that 
$$
\int_{\R} \log(1+|x|) K(x) e^{w(x)} \, dx = \int_{\R} \log(1+|x|) |v(x)|^2 \, dx < +\infty.
$$
Finally, every solution $v \in H^1(\R)$ of \eqref{eq:v_int} must be symmetric-decreasing, i.\,e., it holds that $v=v^*$. 
\end{prop}

Based on the discussion above, we make the following definitions. We let $X$ denote the real Hilbert space of real-valued and even functions on $\R$ given by
$$
X := \left \{ u : \R \to \R \mid  \mbox{$u(-x) = u(x)$ and $\| u \|_{X} < +\infty$}  \right \},
$$
where we define the norm via
$$
 \| u \|_X^2 := \|u \|_{H^1}^2 + \int_{\R}  \log (1+|x|) |u(x)|^2 \, dx.
$$ 
Moreover, we define the set
$$
X^* := \left \{ u \in X \mid u = u^*  \right \}
$$
which is the set of symmetric-decreasing functions that belong to the space $X$. We note that $X^*$ is a {\em closed and convex subset} of $X$.

\subsection{Compactness, A-Priori Bounds and Existence}
 For $\lambda > 0$ and $u \in X^*$ given, we set
$$
T_\lambda(u)(x) := \lambda \sqrt{K(x)} e^{-\frac{1}{2} \int_0^x \Hil(u^2)(y) \, dy}. 
$$
In view of the Proposition \ref{prop:wv_equiv} and the symmetry result in Theorem \ref{thm:main2} (ii), we note that
\be \label{eq:Tv_equiv}
T_\lambda(v) = v \quad \Longleftrightarrow \quad \mbox{$w=\log(K^{-1} v^2)$ solves \eqref{eq:Liouville} with $\int_{\R} K e^w = \int_\R v^2$},
\ee
where, of course, we always assume that $K$ satisfies Assumption \textbf{(A)}. We record the following fact.

\begin{lem}
The map $T_\lambda : X^* \to X^*$ is well-defined and continuous.
\end{lem}

\begin{proof}
We first show that $T_\lambda$ maps $X^*$ into itself. Indeed, let $\lambda >0$ and $u \in X^*$ be given. By Lemma \ref{lem:hilbert}, we deduce that
$$
\psi_u(x) := e^{-\frac{1}{2} \int_0^x \Hil(u^2)(y) \, dy}
$$
is a symmetric-decreasing function with  $0 < \psi_u(x) \leq 1$ for all $x \in \R$. Using Assumption \textbf{(A)} we find that $T_\lambda(u) = \lambda \sqrt{K} \psi_u$ is symmetric-decreasing and belongs to $L^2(\R)$. Next, we easily verify that $\pt_x T_\lambda(u) \in L^2(\R)$ by using Assumption \textbf{(A)} as well as $\Hil(u^2) \in H^1(\R) \subset L^\infty(\R)$.  Finally, by using that $0< \psi_u (x) \leq 1$ again together with the decay properties of $K$, we directly see that $\log (1 + |\cdot|) |T_\lambda(u)|^2 \in L^1(\R)$, whence it follows that $T_\lambda(u) \in X^*$ holds.

Next, we show that $T_\lambda : X^* \to X^*$ is a continuous map. Suppose that $u_k \in X^*$ satisfies $u_k \to u$ in $X$. We readily check that $\psi_{u_k}(x) \to \psi_u(x)$ for almost every $x \in \R$. By dominated convergence and $0 < \psi_{u_k}(x) \leq 1$, we deduce that $T_\lambda(u_k) \to T_\lambda(u)$ in $X$. 
\end{proof}

Next, we establish the following local Lipschitz estimate.

\begin{lem} \label{lem:T_contract}
The map $T_\lambda : X^* \to X^*$ satisfies the estimate
$$
\| T_\lambda(v) - T_\lambda(u) \|_X \lesssim \lambda (\| u \|_X+ \|v \|_X) (1 + \|u\|_X^2 + \|v \|_X^2)  \|u-v \|_{L^2} .
$$

\end{lem}

\begin{proof}
We first establish some auxiliary estimate as follows. For $f \in X^*$, we use again the short-hand notation
$$
\psi_f(x) = e^{-\frac{1}{2} \int_0^x \Hil(f^2)(y) \, dy }.
$$
By Lemma \ref{lem:hilbert} it holds that $0 < \psi_{f}(x) \leq 1$ for any $f \in X^*$. Furthermore, by using the boundedness of the Hilbert transform $\Hil$ on $L^p(\R)$ for any $p \in (1, \infty)$, we deduce, for any $x \in \R$, the pointwise bound
\begin{align*}
\left | \psi_u(x) - \psi_v(x) \right | & \leq \frac{1}{2} \left | \int_0^x  \Hil(u^2) \,dy - \int_0^x \Hil(v^2) \, dy  \right | \leq \frac{1}{2} \int_0^{|x|} \left | \Hil(u^2-v^2)(y) \right | dy \\
& \leq \frac{1}{2} |x|^{1/q} \| \Hil(u^2 -v^2) \|_{L^p} \leq C |x|^{1/q} \| u^2-v^2 \|_{L^p} \\
& \leq C |x|^{1/q} \|u+v \|_{L^r} \| u-v \|_{L^2} \leq C |x|^{1/q} ( \|u\|_X + \| v\|_X) \|u-v \|_{L^2}
\end{align*}
with some constant $C=C(p)>0$ and $1/q+1/p=1$, $1 < p \leq 2$, $1/2+1/r=1/p$, $r > 2$. Since $|x|^{1/q} \sqrt{K} \in L^2(\R)$ by Assumption \textbf{(A)}, provided we take $q \gg 1$ sufficiently large (and thus $p > 1$ is sufficiently close to 1), we obtain
\begin{align*}
\| T_\lambda(u) - T_\lambda(v) \|_{L^2} & \leq C \lambda  ( \|u\|_X + \| v\|_X)  \| |x|^{1/q} \sqrt{K} \|_{L^2} \|u-v \|_{L^2} \\
& \leq C \lambda  ( \|u\|_X + \| v\|_X)   \|u-v \|_{L^2} .
\end{align*}
Likewise, we use $\sqrt{\log (1+ |x)} |x|^{1/q} \sqrt{K} \in L^2(\R)$ for $q \gg 1$ to conclude
$$
\| \sqrt{\log{(1+|\cdot|)}} ( T_\lambda(u) - T_\lambda(v) ) \|_{L^2}  \leq C \lambda  ( \|u\|_X + \| v\|_X)   \|u-v \|_{L^2} .
$$
Next, we notice that 
$$
\pt_x T_\lambda(u) = \lambda \left ( (\pt_x \sqrt{K}) \psi_u - \frac{1}{2} \sqrt{K} \psi_u  \Hil(u^2) \right )  .
$$
Using that $|x|^{1/q} \sqrt{K} \in L^2(\R)$, $\sqrt{K} \in L^\infty(\R)$, and $|\psi_u|, |\psi_v | \leq 1$, we find
\begin{align*}
\left \| \pt_x ( T_\lambda(u) -  T_\lambda(v) ) \right \|_{L^2} & \leq C \lambda \left ( \| (\pt_x \sqrt{K}) ( \psi_u  - \psi_v ) \|_{L^2} + \| \sqrt{K} (\psi_u \Hil(u^2) - \psi_v \Hil(v^2)) \|_{L^2} \right ) \\
& \leq C \lambda ( \|u\|_X + \| v\|_X  ) \left (1 +    \| \Hil(u^2) \|_{L^\infty} + \| \Hil(v^2) \|_{L^\infty}  \right ) \|u-v \|_{L^2} \\
& \quad + C \lambda (\| u \|_X + \|v \|_X  ) \| \Hil(u^2)- \Hil(v^2) \|_{L^2}  \\
& \leq C \lambda (\| u \|_X + \|v \|_X) (1 + \|u \|_X^2 + \| v \|_X^2) \| u-v \|_{L^2}.
\end{align*}
Here we also used that, by Sobolev embeddings and the boundedness of $\Hil$ on $H^1(\R)$ that we have
$$
\| \Hil(f^2) \|_{L^{\infty}} \leq C \| \Hil(f^2) \|_{H^1} \leq C \| f^2 \|_{H^1} \leq C \| f \|_{X}^2
$$
as well as
$$
\| \Hil(u^2)- \Hil(v^2) \|_{L^2} = \| u^2-v^2 \|_{L^2} \leq ( \| u \|_{X} + \|v \|_{X} ) \|u-v \|_{L^2} .
$$
In view of the estimates above, we complete the proof.
\end{proof}

Next, we establish the following result.

\begin{lem}
The map $T_\lambda : X^* \to X^*$ is compact.
\end{lem}

\begin{proof}
Let $(u_k)_{k=1}^\infty$ be a bounded sequence in $X^*$. By Proposition \ref{prop:compact} and passing to a subsequence if necessary, it holds that $(u_k)$ converges strongly in $L^2(\R)$ and thus forms a Cauchy sequence in $L^2(\R)$. From the estimate in Lemma \ref{lem:T_contract}, we deduce that $(T_\lambda(u_k))_{k=1}^\infty$ is a Cauchy sequence in $X$, which converges to some element in $X^*$ due to the closedness of this subset.
\end{proof}

As a next step, we show the following a-priori bound for fixed points of the compact map $T_\lambda : X^* \to X^*$.

\begin{lem} \label{lem:apriori}
For any $v \in X^*$ with $v=T_\lambda(v)$, it holds that
$$
\| T_\lambda(v) \|_X \lesssim  \lambda^2 + \lambda.
$$


\end{lem}

\begin{proof}
Suppose that $v \in X^*$ satisfies $T_\lambda(v) = v$. By the Pohozaev-type result in Lemma \ref{lem:poho}, we deduce the a-priori bound
$$
\| v \|_{L^2}^2 < 2 \pi \lesssim 1.
$$
Next, since $v \in X^*$ is symmetric-decreasing, an application of Lemma \ref{lem:hilbert} yields that $e^{-\frac 1 2 \int_0^x \Hil(v^2)} \leq 1$ for all $x \in \R$. Hence we directly obtain the pointwise bounds
\be \label{ineq:Tv_master}
0 < v(x) \leq \lambda \sqrt{K(x)} \lesssim \lambda \langle x \rangle^{-1/2-\delta},
\ee 
for some $\delta > 0$. From this we readily deduce the bounds
$$
\| v \|_{L^2} \lesssim \lambda, \quad \| \sqrt{\log(1+ |\cdot|)} v \|_{L^2} \lesssim \lambda.
$$

To bound $\pt_x v = \pt_x T_\lambda(v)$, we use the bound $e^{-\frac 1 2 \int_0^x \Hil(v^2)} \leq 1$ once again together with the Gagliardo-Nirenberg interpolation estimate and the $L^2$-boundedness of $\Hil$. This yields 
\begin{align*}
\| \pt_x v \|_{L^2} & \lesssim \| \lambda (\pt_x \sqrt{K}) e^{-\frac{1}{2} \int_0^x \Hil(v^2)} \|_{L^2} + \| v \Hil(v^2) \|_{L^2} \\
& \lesssim \lambda ( \| \pt_x \sqrt{K} \|_{L^2} + \| v \|_{L^\infty} \| v^2 \|_{L^2} )  \lesssim \lambda ( 1 + \| v \|_{L^4}^2) \\
& \lesssim \lambda \left ( 1+ \| v \|_{L^2}^{3/2} \| \pt_x v \|_{L^2}^{1/2} \right )\lesssim \lambda \left (1 + \| \pt_x v \|_{L^2}^{1/2} \right ),
\end{align*}
where in the last step we also used that $\| v \|_{L^2} \lesssim 1$. An elementary argument now yields $\|\pt_x v \|_{L^2} \lesssim \lambda^2 + \lambda$. By recalling the definition of $\| v \|_X$ and the bounds found above, we deduce
$$
\| v \|_X \lesssim \lambda^2 + \lambda,
$$
which completes the proof.
\end{proof}

We are now in the position to show the following existence result.

\begin{prop} \label{prop:exist}
For any $\lambda > 0$, the map $T_\lambda : X^* \to X^*$ has a fixed point. Consequently, for any $v_0 > 0$ there exists a solution $v \in X^*$ of equation \eqref{eq:v_int} with $v(0)=v_0=\lambda \sqrt{K(0)}$. 

Likewise, for any $w_0 \in \R$, there exists a solution $w \in L_{1/2}(\R)$ of \eqref{eq:Liouville} with $K e^w \in L^1(\R)$ and $w(0)=w_0$.
\end{prop}

\begin{proof}
Let $\lambda > 0$ be given. By Lemma \ref{lem:apriori}, there exists a constant $M=M(\lambda) >0 $ such that the following implication holds true:
\be  \label{impl:schauder}
\mbox{$\exists v \in X^*$ and $\exists \sigma \in (0,1]$ with $v=\sigma T_\lambda(v)$} \quad \Rightarrow \quad \| v \|_X < M.
\ee
We consider the closed convex set 
$$
S :=  \{ u \in X^* : \|u  \|_X \leq M\} \subset X
$$
along with the mapping $T^* : S \to S$ defined as
$$
T^*(v) := \begin{dcases*} T_\lambda(v) & if $\|T_\lambda(v) \|\leq M$, \\ M \frac{T_\lambda(v)}{\|T_\lambda(v)\|_X} & if $\| T_\lambda(v) \| \geq M$. \end{dcases*}
$$
Clearly, the map $T^* : S \to S$ is well-defined and continuous. Furthermore, since $T_\lambda(S)$ is precompact by the compactness of $T_\lambda : X^* \to X^*$, we see that the image $T^*(S)$ is precompact in the Banach space $X$. By a suitable version of Schauder's fixed point theorem (see \cite{GiTr-01}[Corollary 11.2]), we conclude the map $T^* : S \to S$ has a fixed point $v \in S$. We claim that $v \in S$ is a fixed point of $T_\lambda$ as well. To show this, let us suppose that $\|T_\lambda(v) \| \geq M$. Then $v = T^*(v) = \sigma T_\lambda(v)$ with $\sigma = M/\| T_\lambda(v) \| \leq 1$. But this contradicts \eqref{impl:schauder}.

Thus we have proven that there exists $v \in S \subset X^*$ such that $T_\lambda(v)=v$, whence $v$ solves \eqref{eq:v_int} with $v(0) = \lambda \sqrt{K(0)}$. Finally, by the discussion in Subsection \ref{subsec:prelim}, the existence of $v$ with $v(0)=v_0>0$ given is equivalent to the fact that $w=\log(K^{-1} v^2) \in L_{1/2}(\R)$ is a solution of \eqref{eq:Liouville} with $w(0) = w_0 = \log(K^{-1}(0) v(0)^2) \in \R$ and $\int_{\R} K e^w \, dx = \int_{\R} v^2  \,dx < +\infty$.
\end{proof}

\subsection*{Proof of Theorem (iii): Existence.} This directly follows from Proposition \ref{prop:exist}. \hfill $\Box$

\section{Uniqueness}

\label{sec:uniqueness} 

In this section, we prove the uniqueness result stated in Theorem \ref{thm:main2}. This will be the main result of this paper.

\subsection{Nondegeneracy and Local Uniqueness}

Recalling the definition of the Banach space $X$ above, we define the map $F : X \times (0,\infty) \to X$ by setting
\be \label{def:Fmap}
 F(u, \lambda)(x) := \lambda \sqrt{K(x)} e^{-\frac{1}{2} \int_0^x \Hil(u^2)(y) \, dy}  - u(x).
\ee
Thus, in terms of the map $T_\lambda$ introduced in the previous section, we can write
$$
F(u, \lambda) = T_\lambda(u) - u .
$$
However, the reader should be aware of the fact that we extend the map $T_\lambda$ from $X^* \subset X$ to all of $X$ here. By standard estimates, it is straightforward to check that the mapping $F : X \times (0,\infty) \to X$ is indeed well-defined and of class $C^1$; see Lemma \ref{lem:F_well} below. By construction, we have the equivalence
\be \label{eq:Fvl}
F(v, \lambda) = 0 \quad \Longleftrightarrow \quad \mbox{$v \in X$ solves \eqref{eq:v_int} with $\displaystyle \lambda = \frac{v(0)}{\sqrt{K(0)}}$}.
\ee

Our next goal is to apply the implicit function theorem in order to construct a locally unique $C^1$-branch $\lambda \mapsto v_\lambda$ around a given solution $(v, \lambda)$ satisfying $F(v,\lambda)=0$. As a key result, we shall need to prove that the Fr\'echet derivative $\pt_v F$ has a bounded inverse on $X$. Indeed, we notice that
\be
\pt_v F(v, \lambda) = \KK - \mathds{1},
\ee 
where $\KK :X \to X$ denotes the bounded linear operator given by
\be
(\KK f)(x) = -v(x) \int_0^x \Hil(v f)(y) \, dy.
\ee
We record the following basic fact.

\begin{lem}\label{lem:KK_compact}
Suppose that $F(v,\lambda) = 0$. Then the linear operator $\KK$ extends to a bounded map from  $\Le(\R)$ into $X$. As a consequence, the linear operator $\KK : \Le(\R) \to \Le(\R)$ is compact.
\end{lem}

\begin{proof}
We show that the linear operator $\KK$ extends to a bounded map from $\Le(\R)$ into $X$ as follows. Let $f \in \Le(\R)$ be given. Using that $v \in H^1(\R)$ and by Sobolev embeddings, we deduce from H\"older's inequality together with the boundedness of $\Hil$ on $L^p(\R)$ when $p \in (1,\infty)$ that we have the pointwise bound
$$
 \left | \int_0^x \Hil(vf)(y) \, dy \right | \leq |x|^{1/q} \| \Hil(v f) \|_{L^p} \leq C |x|^{1/q} \| v f \|_{L^p} \leq C |x|^{1/q}  \| v \|_{H^1} \| f \|_{L^2} 
$$
where $\frac{1}{p} + \frac{1}{q} = 1$, $1 < p  \leq 2$, and with some constant $C=C(p)> 0$. Next, we use that $F(v, \lambda)=0$ holds and thus $v =v^* \in X^*$ is symmetric-decreasing (see Proposition \ref{prop:wv_equiv}). By Lemma \ref{lem:hilbert}, this implies that $e^{-\frac{1}{2} \int_0^x \Hil(v^2) \,dy} \leq 1$ for all $x \in \R$. In particular, this shows that $0 <  v(x) \leq \lambda \sqrt{K(x)}$, which implies the pointwise bound
 \be \label{ineq:KK_point}
 |(\KK f)(x)| \leq C  \| f\|_{L^2} \sqrt{K(x)} |x|^{1/q} \leq C \|f \|_{L^2} \langle x \rangle^{-\frac{1}{2}-\eps},
 \ee
for some $\eps > 0$, where the last inequality follows from the assumed bound for $K$ and by taking $q \gg 1$ sufficiently large (and thus $p>1$ sufficiently close to 1). Clearly, the bound \eqref{ineq:KK_point} shows that 
$$
\int_{\R} |\KK f(x)|^2 \, dx + \int_{\R} \log(1 +|x|) |\KK f(x)|^2 \,dx \leq C \| f\|_{L^2}^2
$$
with some constant $C> 0$ independent of $f$.

Next, by differentiating and using the equation satisfied by $v$, we observe that
\begin{align*}
\| \pt_x \KK f \|_{L^2} & \leq \| (\pt_x v) \int_0^x \Hil(v f) \|_{L^2} + \| v \Hil(v f) \|_{L^2} \\
& \leq C \left (  \| (\pt_x \sqrt{K}) |x|^{1/q} \|_{L^2} + \| \sqrt{K} |x|^{1/q} \|_{L^2} + \| v \|_{L^\infty}^2 \right ) \| f \|_{L^2} \leq C \| f\|_{L^2}
\end{align*}
with some constant $C>0$ independent of $f$. Again, we have chosen $q \gg 1$ sufficiently large and we have used the pointwise bounds for $\sqrt{K}$ and $\pt_x \sqrt{K}$. In summary, we have shown that
$$
\| \KK f \|_{X} \leq C \|f \|_{L^2}
$$
with some constant $C>0$ independent of $f$. Since $v$ and $f$ are even functions, we readily check that $\KK f$ is even as well. Hence we have proven that the linear map $\KK : \Le(\R) \to X$ is bounded.

Finally, we note that the map $\KK : \Le(\R) \to \Le(\R)$ is compact due to fact that the embedding $X \subset \Le(\R)$ is compact; see Proposition \ref{prop:compact} below. 
\end{proof}

Next, we establish the following key result.

\begin{lem} \label{lem:KK_invert}
Let $F(v, \lambda)=0$ hold.  Then the Fr\'echet derivative $\pt_v F = \KK - \mathds{1}$ is invertible on $X$ with bounded inverse.
\end{lem}

\begin{proof}
Since $\KK$ maps $\Le(\R)$ into $X$, it suffices to show $\KK - \mathds{1}$ is invertible on $\Le(\R)$. By the compactness of $\KK$ on $\Le(\R)$ and the Fredholm alternative, this amounts to showing the implication
\be \label{eq:KK_fred}
\KK f = f \quad \mbox{and} \quad f \in \Le(\R) \quad \Rightarrow \quad f = 0.
\ee
Indeed, let us assume that $f \in \Le(\R)$ solves $\KK f = f$. Since $f \in \mathrm{ran} \, \KK$, we obtain that $f \in X$ and in particular the function $f$ is continuous. Next, we note that the equation $\KK f = f$ can be written as
$$
v \psi = f,
$$ 
where we define the even and continuous function $\psi : \R \to \R$ by setting
$$
\psi(x) := -\int_0^x \Hil (v f)(y) \, dy.
$$
Since $\Hil(vf) \in L^2(\R)$, we see that $\pt_x \psi \in L^2(\R)$. On the other hand, we see that $(-\Delta)^{1/2} \psi = \Hil \pt_x \psi = -\Hil^2(vf) = vf \in L^2(\R)$. Therefore, we find that $\psi \in \dot{H}^1_{\mathrm{even}}(\R)$ solves the equation
\be
(-\Delta)^{1/2} \psi - v^2 \psi = 0 \quad \mbox{in $\R$}.
\ee
Now, by using that $W=-v^2$ is $C^1$ and monotone increasing on $[0,\infty)$ and $\psi \in \dot{H}^1_{\mathrm{even}}(\R)$ with $\psi(0)=0$, we obtain that $\psi \equiv 0$ by Lemma \ref{lem:mono} below. Thus $f = v \psi = 0$ is the zero function. This shows \eqref{eq:KK_fred}.

In summary, we have shown that the bounded  linear operator $\pt_v F=\KK-\mathds{1}$ is invertible on $X$. By bounded inverse theorem, its inverse $(\pt_v F)^{-1} : X \to X$ is bounded as well.
\end{proof}

\begin{lem}[Key Lemma] \label{lem:mono}
Let $W : \R \to \R$ be a $C^1$-function with $W'(x) \geq 0$ for $x \geq 0$. Assume that $\psi  \in \dot{H}^1_{\mathrm{even}}(\R)$ solves
$$
(-\Delta)^{1/2} \psi + W \psi = 0 \quad \mbox{in $\R$}
$$
with $W(x) \psi(x)^2 \to 0$ as $x \to +\infty$. Then $\psi(0)=0$ implies that $\psi \equiv 0$.
\end{lem}

\begin{remarks*}
1) The assumption above that $\psi$ is an even function is essential. For example, the odd function $\psi(x) = \frac{2x}{x^2+1} \in H^1(\R)$ (and hence $\psi(0)=0$) solves the equation
$$
(-\Delta)^{1/2} \psi - \frac{2}{1+x^2} \psi = 0 \quad \mbox{in $\R$}.
$$ 

2) We refer to Appendix \ref{sec:mono} for a discussion which relate the arguments below to {\em monotonicity formulas} for $(-\Delta)^s$ in \cite{CaSo-05, CaSi-14,FrLeSi-16}, which involve the $s$-harmonic extension.
\end{remarks*}

\begin{proof}
By integrating the equation on $[0, \infty)$ against $\pt_x \psi \in L^2(\R)$, we find
$$
I + II:= \int_0^{+\infty} ((-\Delta)^{1/2} \psi)(x) \pt_x \psi(x) \, dx + \int_0^{+\infty} W(x) \psi(x) \pt_x \psi(x) \, dx = 0.
$$
We remark that $(-\Delta)^{1/2} \psi, W\psi \in L^2(\R)$ holds and hence the integrals above are absolutely convergent. To analyze the first term, we set $g := \pt_x \psi$ and we notice that
\begin{align*}
I & = \int_0^{+\infty} (\Hil g)(x) g(x) \, dx = \frac{1}{\pi} \int_0^\infty \left ( PV \int_{-\infty}^{+\infty} \frac{g(y)}{x-y} \, dy \right ) g(x) \, dx \\
& = -\frac{1}{\pi} \int_0^{+\infty} \!  \! \int_0^{+\infty} \frac{g(x) g(y)}{x+y} \, dx \, dy,
\end{align*}
where the last step follows by using the anti-symmetry $g(-y) = -g(y)$. Next, we recall the known formula (allowing also for complex-valued functions for the moment):
$$
\int_0^{+\infty} \! \! \int_0^{+\infty} \frac{\overline{\phi}(t) \phi(s)}{t+s} \, ds \, dt = \int_{-\infty}^{+\infty} |(\mathsf{L} \phi)(\lambda)|^2 \, d\lambda \quad \mbox{for all $\phi \in C^\infty_0(\R_+)$},
$$
where $(\mathsf{L} \phi)(\lambda) = \int_0^{+\infty} e^{\lambda t} \phi(t) \, dt$ denotes the one-side Laplace transform; see, e.\,g., \cite{Ya-15}. From this formula, which extends to $L^2(\R_+)$ by density, we readily deduce the classical fact the the Carleman--Hankel operator with the kernel $(x+y)^{-1}$ on $L^2(\R_{+})$ is positive definite.  Hence, if we go back to the expression for $I$, we deduce that
$$
I \leq 0 \quad \mbox{with $I=0$ if and only if $g \equiv 0$}.
$$

On the other hand, by integration by parts and using that $\psi(0)=0$ and $W(x) \psi(x)^2 \to 0$ as $x \to +\infty$, we obtain from $W'(x) \geq 0$ for $x \geq 0$ that
$$
II  = \int_0^{+\infty} W(x) \psi(x) \pt_x \psi(x) \, dx = - \frac{1}{2} \int_0^{+\infty} W'(x) \psi(x)^2 \, dx \leq 0.
$$
Because of $I+II=0$, we conclude that $I=0$ must hold and thus $g=\pt_x \psi \equiv 0$. Hence $\psi$ is a constant function. Since $\psi(0)=0$ by assumption, this implies that $\psi \equiv 0$.
\end{proof}

By applying the implicit function theorem together with Lemma \ref{lem:KK_invert}, we can construct a unique local branch around any given solution of the equation $F(v,\lambda)=0$ as follows.

\begin{prop} \label{prop:local}
Let $(v, \lambda) \in X \times (0, \infty)$ solve $F(v, \lambda)=0$. Then there exist an open interval $I = (\lambda-\eps, \lambda+\eps) \cap (0, \infty)$ with some $\eps = \eps(\| v \|_{X})> 0$ and a $C^1$-map 
$$
I \to X, \quad t \mapsto v_t
$$ 
such that $v_{t=\lambda}=v$ and $F(v_t, t)=0$ for all $t \in I$. Moreover, there exists a neighborhood $N \subset X$ around $v$ such that $F(u, t)=0$ with $(u,t) \in N \times I$ implies that $u = v_t$.
\end{prop}

Below, we will see that we can extend every such local branch $t \mapsto v_t$ to all of $t \in (0, \lambda]$ thanks to a-priori bounds.

\subsection{Global Uniqueness and Proof of Theorem \ref{thm:main2}}

First, we notice that we must have global uniqueness of solutions of $F(v, \lambda)=0$ for sufficiently small $\lambda> 0$.

\begin{prop} \label{prop:small_unique}
There exists $0 < \lambda_* \ll 1$ such that the solution $v \in X$ of $F(v,\lambda)=0$ is unique for $0 < \lambda \leq \lambda_*$.
\end{prop}

\begin{proof}
Let $u, v \in X$ both solve $F(u,\lambda) = F(v,\lambda)=0$. Since $u,v \in X$ are symmetric-decreasing by Proposition \ref{prop:wv_equiv}, we see that  $u$ and $v$ are fixed points of the map  $T_\lambda : X^* \to X^*$. if we combine the a-priori bounds in Lemma \ref{lem:apriori} with Lipschitz estimates in Lemma \ref{lem:T_contract}, we deduce that
$$
\| u-v \|_{L^2} \lesssim  \lambda ( \lambda^6 + \lambda ) \| u-v \|_{L^2} .
$$
Thus there exists $0 < \lambda_* \ll 1$ sufficiently small such that $\lambda \in (0, \lambda_*]$ implies $\| u - v \|_{L^2} \leq \frac{1}{2} \|u -v \|_{L^2}$ and hence $u \equiv v$.
\end{proof}

We are now ready to prove Theorem \ref{thm:main2} as follows. Let $w, \tilde{w} \in L_{1/2}(\R)$ be two solutions of \eqref{eq:Liouville} with $K e^w, K e^{\tilde{w}} \in L^1(\R)$. Correspondingly, we define the functions $v = \sqrt{K e^w} \in X$ and $\tilde{v} = \sqrt{K e^{\tilde{w}}} \in X$. Suppose now that $w(0) = \tilde{w}(0)$ holds, which implies $v(0)= \tilde{v}(0)$. By the previous discussion, we have $F(v,\lambda)=F(\tilde{v},\lambda)=0$ with $\lambda = v(0)/\sqrt{K(0)} >0$. By Proposition \ref{prop:local} and thanks to the a-priori bounds in Lemma \ref{lem:apriori}, we can construct two branches 
$$
t \mapsto v_t \quad \mbox{and} \quad t \mapsto \tilde{v}_t \quad \mbox{for all $t \in (0, \lambda]$}
$$ 
such that $F(v_t,t)=F(\tilde{v}_t,t)=0$ for $t \in (0, \lambda]$ and $v_{t=\lambda}=v$ and $\tilde{v}_{t=\lambda}=\tilde{v}$. Suppose now $w \neq \tilde{w}$ and hence $v \neq \tilde{v}$. By the local uniqueness property in Proposition \ref{prop:local}, the branches can never intersect, i.\,e., we have $v_t \neq \tilde{v}_t$ for all $t \in (0, \lambda]$. But this contradicts the uniqueness result in Proposition \ref{prop:small_unique} whenever $0 < t \ll 1$ is sufficiently small. Therefore, we conclude that $v = \tilde{v}$ and hence $w=\log(K^{-1} v^2) = \log(K^{-1} \tilde{v}^2) = \tilde{w}$ holds true. 

The proof of Theorem \ref{thm:main2} is now complete. \hfill $\Box$

\begin{appendix}

\section{Some Technical Facts}

\begin{lem} \label{lem:hilbert}
Let $f \in X$ be symmetric-decreasing, i.\,e., we have $f = f^*$. Then its Hilbert transform satisfies
$$
\Hil(f)(x) \geq 0 \quad \mbox{for}  \quad x \geq 0 \quad \mbox{and} \quad \Hil(f)(x) \leq 0 \quad \mbox{for} \quad x \leq 0.
$$
\end{lem}

\begin{proof}
The proof is elementary. Since $f = f^*$ is an even function, we note that its Hilbert transform
$$
\Hil(f)(x) = \frac{1}{\pi} PV \int_{-\infty}^{+\infty} \frac{f(y)}{x-y} \, dy = \frac{1}{\pi} \lim_{\eps \to 0} \int_\eps^{+\infty} \frac{f(x-y)-f(x+y)}{y} \, dy,
$$
is an odd function in $x \in \R$. Thus it suffices to show that $\Hil(f)(x) \geq 0$ for $x \geq 0$.  Since $|x-y| \leq |x+y| = x+y$ for $x,y \geq 0$ and $f=f^*$ is symmetric-decreasing, we see that $f(x-y) \geq f(x+y)$ for all $x,y \geq 0$. Thus the claim directly follows from the integral expression for $\Hil(f)(x)$.
\end{proof}

%

\begin{prop} \label{prop:compact}
The embedding $X \subset \Le(\R)$ is compact.
\end{prop} 

\begin{proof}
Suppose that $(f_n)$ is a bounded sequence in $X$. In particular, the sequence $(f_n)$ is bounded in $H^1(\R)$. Hence, by local Rellich compactness, we can assume that $(f_n)$ converges in $L^2_{\mathrm{loc}}(\R)$ after passing to a subsequence if necessary. Next, by the uniform bound $\int_{\R} \log(1 +|x|) |f_n(x)|^2 \,dx \leq C$ with some constant $C> 0$ independent of $n$, the property $\log(1 +|x|) \to +\infty$ as $|x| \to +\infty$ readily implies that $(f_n)$ strongly converges in $L^2(\R)$. Since $\Le(\R)$ is a closed subspace of $L^2(\R)$, the proof is complete.
\end{proof}

\begin{lem} \label{lem:F_well}
The map $F : X \times (0, \infty) \to X$  given in \eqref{def:Fmap} is well-defined and of class $C^1$.
\end{lem}

\begin{proof}
Let $(u,\lambda) \in X \times (0,\infty)$ be given. First we show that $F(u,\lambda) \in X$ as follows. We set
$$
h_u(x) := -\int_0^x \Hil(u^2)(y) \, dy.
$$
We claim that 
\be \label{ineq:h_u}
h_u(x) \leq A \| u \|_{X}^2 \quad \mbox{for all $x \in \R$},
\ee
with some universal constant $A > 0$. Indeed, using that $\Hil \pt_x = (-\Delta)^{1/2}$, we note that $(-\Delta)^{1/2} h_u = u^2$ in $\R$. By adapting the arguments in the proof of Lemma \ref{lem:w_int_log} and using that $\log (1 + |\cdot|) |u|^2 \in L^1(\R)$, we deduce that
$$
h_u(x) = -\frac{1}{\pi} \int_{\R} \log |x-y| u(y)^2 \, dy + C_0
$$
with some constant $C_0 \in \R$. Since $h_u(0)=0$, we find the upper bound
$$
C_0 = \frac{1}{\pi} \int_{\R} \left ( \log |y| \right ) u(y)^2 \, dy \leq \frac{1}{\pi} \int_\R \log (1+|y|) u(y)^2 \,dy \leq \frac{1}{\pi} \| u \|_{X}^2.
$$
Furthermore, we estimate
\begin{align*}
-\frac{1}{\pi} \int_\R \log |x-y| u(y)^2 \, dy & \leq -\frac{1}{\pi} \int_{B_1(x)} \log |x-y| u(y)^2 \, dy  \\ 
& \leq \frac{1}{\pi} \left \| \log |x-\cdot | \right \|_{L^2(B_1(x))} \| u^2 \|_{L^2} \leq B \| u \|_{X}^2
\end{align*}
with some constant $B>0$. This completes the proof of \eqref{ineq:h_u}. 

Using \eqref{ineq:h_u}, we see that $e^{h_u} \in L^\infty(\R)$ and thus, by our assumptions on $K$, we deduce that $F(u,\lambda) = \lambda \sqrt{K} e^{\frac{1}{2} h_u} - u$ satisfies
$$
\int_{\R} |F(u,\lambda)(x)|^2 + \int_{\R} \log(1+|x|) |F(u,\lambda)(x)|^2 \, dx < +\infty.
$$
Similarly, we show that $\pt_x F(u,\lambda) \in L^2(\R)$. Finally, it is easy to see that $F(u,\lambda)(-x) = F(u,\lambda)(x)$ using that $u(x) = u(-x)$ holds for $u \in X$. This proves that $F(u,\lambda) \in X$.

The fact that the map $F:X \times (0,\infty) \to X$ is continuous follows by using dominated convergence together with previous bounds, standard estimates, and the fact that $u_n \to u$ in $X$ implies that $h_{u_n} \to h_u$ pointwise almost everywhere. Furthermore, it is straightforward to verify that $F$ is of class $C^1$. We omit the details.
\end{proof}

\section{Relation to the Monotonicity Formula} 
\label{sec:mono}

In this section, we explain how the arguments in the proof of the key Lemma \ref{lem:mono} can be seen from the perspective of \textbf{monotonicity formulas} for  the fractional Laplacian $(-\Delta)^s$ in \cite{CaSo-05,CaSi-14, FrLeSi-16}.

To simplify the presentation and to focus on the main ideas, we will work on a purely calculational level and we thus omit any technicalities. Furthermore, we will exclusively consider $n=1$ space dimension and $(-\Delta)^s$ with $s=1/2$ (which is the relevant case in this paper). Suppose that $u : \R \to \R$ solves
\be \label{eq:u_mono}
(-\Delta)^{1/2} u + V u = 0 \quad \mbox{in $\R$},
\ee
where $V: \R \to \R$ is a given potential of class $C^1$, say. Following \cite{FrLeSi-16}  we introduce the function $H$ on $\R$ defined as
\be 
H(x) = \frac{1}{2} \int_0^\infty \left \{ u_x(x,t)^2 - u_t(x,t)^2 \right \} dt - \frac{1}{2} V(x) u(x)^2.
\ee
Here $u(x,t)$ denotes the harmonic extension of $u(x)$ to the upper half-plane $\R_+^2 = \R \times \{ t > 0 \}$. Using the classical fact $-\pt_t u(0,x)  = (-\Delta)^{1/2} u(x)$ and \eqref{eq:u_mono}, a calculation yields that
\begin{equation} \label{eq:mono_H}
H'(x) = -\frac{1}{2} V'(x) u(x)^2 .
\end{equation}
Thus if $V$ is monotone increasing, we see that $H'(x) \leq 0$ for $x \geq 0$, showing that $H$ is a monotone decreasing quantity on the half-line $[0,\infty)$. See \cite{FrLeSi-16} for applications to show uniqueness results.

To make a link with $H(x)$ to the arguments in the proof of Lemma \ref{lem:mono}, let us consider the expression 
\be
\tilde{H}(x) = \int_x^\infty ((-\Delta)^{1/2} u(y)) u_x(y) \, dy  + \int_{x}^\infty V(y) u(y) u_x(y) \, dy  .
\ee
Note that we have $\tilde{H}(x) \equiv 0$, since $(-\Delta)^{1/2} u + V u = 0$ holds. But by separately analyzing the two integrals above, we find the following relation to $H(x)$. First, we claim
\be  \label{eq:Gauss}
\int_x^{\infty} ((-\Delta)^{1/2} u(y)) u_x(y) \, dy = \frac{1}{2} \int_0^\infty  \left \{ u_x(x,t)^2 - u_t(x,t)^2 \right \} dt 
\ee 
for any $x \in \R$. To see this, we introduce the vector field $\mathbf{F} : \R_+^2 \to \R^2$ with
$$
\mathbf{F}(x,t) =  \left ( \begin{array}{c}  {-u_t} u_x \\ \frac{1}{2} ( u_x^2 - u_t^2) \end{array} \right ),
$$
where $u=u(x,t)$ denotes the harmonic extension $u$ to $\R^2_+$.  In view of $u_{xx}+u_{tt}=0$ in $\R^2_+$, we readily check that $\mathbf{F}: \R^2_+ \to \R^2$ is curl-free, i.\,e.,
$$
\mathrm{curl} \, \mathbf{F} = \pt_x F_2 - \pt_t F_1 = 0.
$$
By Stokes' theorem and assuming sufficient decay of $\mathbf{F}$ at infinity, we can deduce 
$$
0 = \int \! \! \int_{D} \mathrm{curl} \, \mathbf{F} \, dx \, dt = \oint_{\pt D} \mathbf{F} \cdot d \mathbf{s} = \int_x^\infty F_1(y,0) \, dy - \int_0^\infty F_2(x,t) \, dt 
$$
with the region $D= [x,\infty) \times [0,\infty) \subset \overline{\R^2_+}$. Because of $F_1 \big |_{t=0} = ((-\Delta)^{1/2}u) u_x$, we see that \eqref{eq:Gauss} holds. On the other hand, by integration by parts and assuming that $V u^2$ vanishes at infinity, we immediately find
$$
\int_x^{\infty} V(y) u(y) u_x(y) \, dy = -\frac{1}{2} V(x) u(x)^2 - \frac{1}{2} \int_x^\infty V'(y) u(y)^2 \, dy
$$
In summary, we have derived that following identity
\be
\tilde{H}(x) = H(x) - \frac{1}{2} \int_x^\infty V'(y) u(y)^2 \, dy 
\ee
relating $\tilde{H}$ and $H$. In view of $\tilde{H}(x) \equiv 0$, we obtain \eqref{eq:mono_H} as a direct consequence. 

\end{appendix}

\end{document}